\documentclass[reqno,10pt]{amsart}
\usepackage[english]{babel}
\usepackage{amsthm,hyphenat}
\usepackage{amsmath}
\usepackage{amssymb}
\usepackage[latin1]{inputenc}
\usepackage[all]{xy}
\usepackage{dsfont}
\usepackage[bookmarksnumbered,colorlinks]{hyperref}
\usepackage{graphics}
\usepackage{enumerate}
\def\bb#1\eb{\textcolor{blue}
{#1}} %
\newcommand{\R}{\mathds R}

\title[Chern connection as a family of affine connections]{Chern connection of a  pseudo-Finsler metric as a family of affine connections}

\author[M. A. Javaloyes]{Miguel Angel Javaloyes}
\address{Departamento de Matem\'aticas, \hfill\break\indent
Universidad de Murcia, \hfill\break\indent
Campus de Espinardo,\hfill\break\indent
30100 Espinardo, Murcia, Spain}
\email{majava@um.es}

\dedicatory{This article is dedicated to Professor Lajos Tamassy on the occasion of his 90th birthday}

\date{03.04.2011}

\thanks{This work was partially supported by MINECO (Ministerio de Econom\'{\i}a y Competitividad) and FEDER (Fondo Europeo de Desarrollo Regional) project MTM2012-34037 and Fundaci\'{o}n S\'{e}neca project 04540/GERM/06, Spain. This research is a
result of the activity developed within the framework of the Programme in Support of Excellence Groups of the Regi\'{o}n de Murcia, Spain, by Fundaci\'{o}n S\'{e}neca, Regional Agency for Science and Technology (Regional Plan for
Science and Technology 2007-2010).}

\thanks{2000 {\em Mathematics Subject Classification:} Primary  53C50, 53C60\\
\textbf{Key words:} Finsler metrics, Chern connection, Flag curvature.}

\begin{document}
\newtheorem{thm}{Theorem}[section]
\newtheorem{prop}[thm]{Proposition}
\newtheorem{lemma}[thm]{Lemma}
\newtheorem{cor}[thm]{Corollary}
\theoremstyle{definition}
\newtheorem{defi}[thm]{Definition}
\newtheorem{notation}[thm]{Notation}
\newtheorem{exe}[thm]{Example}
\newtheorem{conj}[thm]{Conjecture}
\newtheorem{prob}[thm]{Problem}
\newtheorem{rem}[thm]{Remark}

\begin{abstract}
We consider the Chern connection of a (conic) pseudo-Finsler manifold $(M,L)$ as a linear connection
$\nabla^V$ on any open subset $\Omega\subset M$ associated to any vector field $V$ on $\Omega$ which is non-zero everywhere. This connection is torsion-free and almost metric compatible with respect to the fundamental tensor $g$. Then we show some properties of the curvature tensor $R^V$ associated to $\nabla^V$ and in particular we prove that the Jacobi operator of $R^V$ along a geodesic coincides with the one given by the Chern curvature. 
\end{abstract}

\maketitle
\section{Introduction}
The Chern connection of a Finsler metric $F$ on a manifold $M$ was originally conceived by S.-S. Chern \cite{Chern43} as a connection in a fiber bundle over $TM\setminus\bf 0$ and introduced again independently by H. Rund in \cite{Rund59} (see also \cite{An95}). Then it was completely forgotten until the work of D. Bao and S.-S. Chern \cite{BaCh93}, where the authors show the extraordinary usefulness of the Chern connection in treating global problems of Finsler geometry. In particular, the connection provides an easy way to compute the flag curvature of a Finsler metric, which is an important invariant  associated to the deviation of geodesics. But when considered as a connection in a fiber bundle over $TM \setminus\bf 0$, it does not allow one to use the coordinate-free global methods of Modern Differential Geometry employed in the study of Riemannian Geometry. This can be overcome by using the osculating Riemannian metric associated to a Finsler metric introduced by A. Nazim in his Ph. D. thesis  \cite{nazim36} and studied sistematically  by O. Varga \cite{Varga41}. More precisely, for any $p\in M$ and any 
non-zero vector $v$ in $T_pM$, the fundamental tensor provides a scalar product in $T_pM$. In particular this idea was developed by 
H.-H.~Matthias in his Ph. D. Thesis \cite[Def. 2.5]{Mat80} to define an affine connection $\nabla^V$ on an open subset $\Omega\subset M$ for every vector field $V$ on $\Omega$ which is non-zero everywhere. The connection $\nabla^V$ is torsion-free and almost $g$-compatible, meaning that the derivative of the osculating Riemannian metric $g_V$ is not zero, but a certain expression in terms of the Cartan tensor (see \eqref{cartantensor} and Definition~\ref{cartanconnection}). The approach of H.-H. Matthias was collected in \cite[page 100]{Sh01}, where the author shows a relation of the Jacobi operator of the metric $g_V$ in case that $V$ is a geodesic field \cite[Proposition 8.4.3 and Lemma 8.1.1]{Sh01} and recovered again by other authors as H-B. Rademacher \cite{Rad04b,Rad04a} and Z. Kovacs and A. Toth in \cite{ToKo08} and also used by J. C. \'Alvarez Paiva and C. E. Dur\'an in \cite[Theorem 6.1]{AlDur01}. 

None of the cited works makes a detailed study of the properties of the curvature tensor $R^V$ of $\nabla^V$ and its relation with the flag curvature when $V$ is not a geodesic field. Our main goal is to write down the symmetries and basic properties of $R^V$ in order to establish the relation of $R^V$ with the flag curvature in the general case when $V$ is not a geodesic field. This result can be used for example to obtain the first and the second variation of the energy functional with coordinate-free global methods (see \cite{JavSoares13}).

The work is structured  as follows.  In Section \ref{quadratic} we introduce the notion of  pseudo-Finsler metric, which generalizes the former notions of Finsler metric in the sense that the function is not necessarily positive and it is positive homogeneous of degree two, rather than one, with non-degenerate fundamental tensor.  Then we introduce the Cartan tensor associated to $(M,L)$ and an affine connection $\nabla^V$  associated to a vector field $V$ in an open subset $\Omega\subset M$ which takes values in $A\cap T\Omega$ . This connection is characterized as the unique one which is torsion-free and almost metric compatible
(see Definition~\ref{cartanconnection}). As it is shown in Proposition ~\ref{curveconnection}, the connection $\nabla^V$ can be identified in a certain sense with the Chern connection and it defines a covariant derivative $D_\gamma^V$ along any curve $\gamma$ with a reference vector $V$ along the curve which is non-zero everywhere. 

 In Section \ref{curvature}
we firstly study in Proposition~\ref{propcurvature} the symmetric properties of the curvature tensor $R^V$ of $\nabla^V$. In Subsection \ref{flag} we establish the link between the tensor $R^V$ and the flag curvature of $(M,L)$. Unlike $\nabla^V$, the curvature tensor $R^V$ depends not only on the value of $V$ in $p\in M$, but in the whole vector field in a neighborhood of $p$. Nevertheless we show that the Jacobi operator can be defined along a curve (see Proposition \ref{curvaturacurva}) and it coincides with the Jacobi operator obtained from the curvature of the Chern connection as a connection on the fiber bundle $\pi^*_A(TM)$ over the conic subset $A$ when $\gamma$ is a geodesic (see Theorem~\ref{curvaturaway}).
This allows us to compute the flag curvature in terms of $R^V$ (see Corollary \ref{flagcur}).

\section{Pseudo-Finsler metrics}\label{quadratic}
Let $M$ be a smooth manifold, $TM$ its tangent bundle and $\pi:TM\rightarrow M$ the natural projection. We will say that an open subset $A\subset TM$ is {\it conic} if for every $v\in A$ and $\lambda>0$, we have $\lambda v\in A$. We say that a function $L:A\subset TM\rightarrow \R$ is a {\it (two-homogeneous, conic) pseudo-Finsler metric} if it is positive homogeneous of degree 2, that is, $L(\lambda v)=\lambda^2 L(v)$ for every $v\in A$ and $\lambda>0$, and the {\it fundamental tensor} of $L$ defined as 
\[g_v(u,w):=\frac 12 \frac{\partial^2}{\partial t\partial s} L(v+tu+sw)|_{t=s=0},\]
for any $v\in A$ and $u,w\in T_{\pi(v)}M$, is non-degenerate (see \cite{JS13} for explicit computations of the fundamental tensor in some important cases). 
In the following, we will assume that the pseudo-Finsler metric is two-homogeneous and conic, namely, not necessarily defined in the whole tangent bundle. Observe that in some references, a pseudo-Finsler metric is defined as a one-homogeneous function \cite{JS13} with possibly degenerate fundamental tensor. The square of such a function fits in our definition whenever the fundamental tensor is non-degenerate.

Then we define the {\it Cartan tensor} of $L$ as the trilinear form
\begin{equation}\label{cartantensor}
C_v(w_1,w_2,w_3)=\frac 14\left. \frac{\partial^3}{\partial s_3\partial s_2\partial s_1}L\left(v+\sum_{i=1}^3 s_iw_i\right)\right|_{s_1=s_2=s_3=0},
\end{equation}
for any $v\in A$ and $w_1,w_2,w_3\in T_{\pi(v)}M$. It is easy to see that $C_v$ is homogeneous of degree $-1$ in $v$ and 
\begin{equation}\label{porvenir}
C_v(v,w_1,w_2)=C_v(w_1,v,w_2)=C_v(w_1,w_2,v)=0
\end{equation}
for any $v\in A$ and $w_1,w_2\in T_{\pi(v)}M$ (see for example \cite[Proposition 2.6 and Remark 2.9]{JavSoares13}).

\subsection{Chern connection} Given a pseudo-Finsler manifold $(M,L)$ with conic domain $A\subset TM\setminus {\bf 0}$, we will say that a vector field $V$ on an open subset $\Omega\subset M$ is {\it $L$-admissible} if $V(x)\in A\cap T_xM$ for every $x\in \Omega$.

\begin{defi}\label{cartanconnection}
Let $(M,L)$ be a pseudo-Finsler manifold and $V$ an $L$-admissible vector field on an open subset $\Omega\subset M$. Consider an affine connection $\nabla^V$ on  $\Omega$ and denote by ${\mathfrak X}(\Omega)$ the space of vector fields on $\Omega$. We say that
\begin{enumerate}
 \item $\nabla^V$ is {\it torsion-free} if
$\nabla^V_XY-\nabla^V_YX=[X,Y]$
for every $X, Y\in{\mathfrak X}(\Omega)$,
\item $\nabla^V$ is {\it almost $g$-compatible} if
\[X( g_V(Y,Z))=g_V(\nabla^V_XY,Z)+g_V(Y,\nabla^V_XZ)+2 C_V(\nabla^V_XV,Y,Z),\]
where $X,Y,Z\in{\mathfrak X}(\Omega)$ and 
$g_V$ and $C_V$ are, respectively, the fundamental tensor and the Cartan tensor of $L$  evaluated on the vector field $V$.
\end{enumerate}
\end{defi}
\begin{rem}
Observe that the condition of almost $g$-compatibility for the pseudo-Finsler metric $L$  given above is equivalent to the equation
\[\nabla^V_X (g_V)(Y,Z)=2 C_V(\nabla^V_XV,Y,Z),\]
namely, the derivative of $g_V$ is expressed in terms of Cartan tensor.
\end{rem}
\begin{prop}\label{existconn}
A pseudo-Finsler manifold $(M,L)$ and an $L$-admissible vector field $V$ on an open subset $\Omega\subset M$ 
admit a unique torsion-free and almost $g$-compatible affine  connection $\nabla^V$.
\end{prop}
\begin{proof}
Observe that $\nabla^V$ is determined by a ``Koszul formula'' as
\begin{multline*}
 2 g_V(\nabla^V_XY,Z)= X (g_V(Y,Z))-Z (g_V(X,Y))+Y (g_V(Z,X))\\
+g_V([X,Y],Z)+g_V([Z,X],Y)-g_V([Y,Z],X)\\
+2( -C_V(\nabla^V_XV,Y,Z)-C_V(\nabla^V_YV,Z,X)+C_V(\nabla^V_ZV,X,Y)).
\end{multline*}
Indeed, when $X=Y=V$, the terms of the Cartan tensor vanish because of \eqref{porvenir}
determining $\nabla^V_VV$ and then $\nabla^V_XV$ can be determined using  $\nabla^V_VV$.
Moreover, it follows from Koszul formula that $\nabla^V_XV$ is $f$-linear in $X$, that is, $\nabla^V_{fX}V=f \nabla^V_XV$ for any real function $f$ on $U$. Then it is clear that Koszul formula determines $\nabla^V_XY$ for any vector fields $X,Y$ on $U$ and it is an affine linear connection.
\end{proof}
Analogous computations to those of the last proof can be found in \cite[Theorem 3.10]{Rad04b}.
\begin{rem}\label{nablahomo}
Observe that $\nabla^V$ is homogeneous of degree zero in $V$ in the sense that if $\lambda>0$, then $\nabla^{\lambda V}=\nabla^V$, since 
$\nabla^{\lambda V}$ solves the same equations as $\nabla^V$ (see Definition \ref{cartanconnection}).
\end{rem}
Let us denote by $n$ the dimension of $M$. Now fix a coordinate system on an open subset $\Omega$ of $M$, that is, a map $\varphi:\Omega\rightarrow \varphi(\Omega)\subset \R^n$, given by
$\varphi(p)=(x^1(p),x^2(p),\ldots,x^n(p))$ for every $p\in\Omega$ and denote as $\frac{\partial}{\partial x^1},\ldots,\frac{\partial}{\partial x^n}$, the vector fields associated to the system, that is, the partial derivatives of $\varphi^ {-1}$ (composed with $\varphi$ in order to have vector fields in $\Omega$), which we will
call the {\it coordinate basis associated to $\varphi$}. We define \emph{the formal Christoffel symbols} associated to $\varphi$ and the vector field $V$, $\Gamma_{\,\,ij}^k(V) $, by means of the equation
\[
\nabla^V_{\frac{\partial}{\partial x^i}}\left(\frac{\partial}{\partial x^j}\right)=\sum_{k=1}^ n\Gamma_{\,\,ij}^k(V) \frac{\partial}{\partial x^k},
\]
for $i,j=1,\ldots,n$.
\begin{rem}
Let us denote by $g_{ij}(v)=g_v(\frac{\partial}{\partial x^i},\frac{\partial}{\partial x^j})$ the functions defined for any $v\in T\Omega\cap A$. Moreover, $g^{ij}$ will be the coefficients of the inverse matrix of $\{g_{ij}\}$ with $i,j=1,\ldots,n$. From now on we will use the {\it Einstein summation convention} consisting in omitting the sums from 1 to $n$  when an index appears up and down, and we will raise and lower indices using $g_{ij}$ and $g^{ij}$, for example
\[\Gamma_{kij}(V)=\sum_{m=1}^n g_{km}(V) \Gamma_{\,\,\,ij}^m(V)=g_{km}(V) \Gamma_{\,\,\,ij}^m(V),\]
for any $L$-admissible vector field $V$ on $\Omega$. Moreover,  $j$ in $\frac{\partial}{\partial x^j}$ will be considered a down index and then 
\[\sum_{k=1}^ n\Gamma_{\,\,ij}^k(V) \frac{\partial}{\partial x^k}=\Gamma_{\,\,ij}^k(V) \frac{\partial}{\partial x^k}.\]
\end{rem}
In principle, $\Gamma_{\,\,ij}^k(V)$ depends on the vector field $V$, but let us see that this is not the case and in fact they are homogeneous real  functions of degree zero on $A\cap T\Omega$ (see Remark \ref{nablahomo}).  

Given a curve $\gamma:[a,b]\rightarrow M$, we will define the vector bundle $\gamma^*(TM)$ as the vector bundle over $[a,b]$ induced by $\pi:TM\rightarrow M$ through $\gamma$. The smooth sections of $\gamma^*(TM)$ are called vector fields along $\gamma$ and we will denote by ${\mathfrak X}(\gamma)$ the subset of such smooth sections.
We will say that $V\in {\mathfrak X}(\gamma)$ is \emph{ $L$-admissible} if $V(t)\in A$ for every $t\in [a,b]$.
\begin{prop}\label{curveconnection}
Let $(M,L)$ be a pseudo-Finsler manifold  and $V$ an $L$-ad\-mi\-ssi\-ble vector field  in an open subset $\Omega\subset M$ endowed with a system of coordinates $\varphi$. Then the Christoffel symbols of $\nabla^V$ depend only on $v=V(x)$, with $x\in \Omega$, and not on the extension of $v$. Moreover,  they coincide with the Christoffel symbols of the Chern connection (see \cite[Eq. (2.4.9)]{BaChSh00}). Given a smooth curve $\gamma:[a,b]\rightarrow M$, $X\in{\mathfrak X}(\gamma)$ and $W$ an $L$-admissible vector field along $\gamma$, we can define the covariant derivative of $X$ along $\gamma$  having $W$ as reference vector  as 
\begin{equation}\label{connection}
D^W_\gamma X= \frac{{\rm d}X^i}{{\rm d}t}\frac{\partial}{\partial x^i}+  X^i(t)\dot\gamma^j(t)\Gamma_{\,\,ij}^k(W(t))\frac{\partial}{\partial x^k},
\end{equation}
where $(X^1,\ldots,X^n)$ and $(\dot\gamma^1,\ldots,
\dot\gamma^n)$ are respectively the coordinates of $X$ and $\dot\gamma$ in the coordinate basis of $\varphi$. Moreover, it is almost $g$-compatible, namely,
if $X,Y\in {\mathfrak X}(\gamma)$, then
\begin{equation}\label{almostmetric}
\frac{d}{dt}g_{W}(X,Y)=g_{W}(D_\gamma^WX,Y)+g_{W}(X,D_\gamma^WY)+2C_{W}(D_\gamma^WW,X,Y).
\end{equation}
\end{prop}
\begin{proof}
Let us observe that the functions $g_{ij}$ are defined in  $A\cap T\Omega$ and we will consider the natural coordinate system in $T\Omega$ associated to $x^1,x^2,\ldots,x^n$, which will be denoted as $x^1,x^2,\ldots,x^n,y^1,y^2,\ldots,y^n$. Denote by $V^1,\ldots,V^n$
the coordinates of $V$ in $(\Omega,\varphi)$. Observe that 
\begin{equation}\label{abuse}
\frac{\partial (g_{ij}\circ \varphi^{-1})(x,V^1(x),\ldots,V^n(x))}{\partial x^k}=\frac{\partial g_{ij}}{\partial x^k}+\frac{\partial V^l}{\partial x^k}\frac{\partial g_{ij}}{\partial y^l}=\frac{\partial g_{ij}}{\partial x^k}+2\frac{\partial V^l}{\partial x^k}C_{lij},
\end{equation}
where $C_{lij}=C_V(\frac{\partial}{\partial x^l},\frac{\partial}{\partial x^i},\frac{\partial}{\partial x^j})$ and $x=(x^1,\ldots,x^n)\in \varphi(\Omega)$. With abuse of notation we have omitted the composition with $\varphi^{-1}$ and the point of evaluation  in the right-hand terms and in the rest of the proof.
Now from Koszul formula for $X=\frac{\partial}{\partial x^i}$, $Y=\frac{\partial}{\partial x^j}$ and $Z=\frac{\partial}{\partial x^k}$ and \eqref{abuse}, we obtain
\begin{equation}\label{firstexpr}
\Gamma_{kji}=\gamma_{kji}-V^l\Gamma_{\,\,il}^p C_{pkj}-V^l\Gamma_{\,\,jl}^p C_{pik}+V^l\Gamma_{\,\,kl}^p C_{pji},
\end{equation}
where 
\[\gamma_{ijk}=\frac 12\left(\frac{\partial g_{ij}}{\partial x^k}-\frac{\partial g_{jk}}{\partial x^i}+\frac{\partial g_{ki}}{\partial x^j}\right).\]
 Observe that 
by \ref{porvenir}, $ V^l C_{lij}=0$ and $C_{ijk}$ is symmetric in the three indexes. Then
$V^iV^j\Gamma_{kji}= V^iV^j\gamma_{kij}$
and raising indices we get
\begin{equation}\label{vivj2}
 V^iV^j\Gamma^k_{\,\, ji}= V^iV^j\gamma^k_{\,\, ij}.
\end{equation}
 From \eqref{firstexpr} and using \eqref{vivj2} we conclude that
\begin{equation}\label{vi}
 V^i\Gamma^s_{\,\, ji}= g^{ks}V^i\Gamma_{kji}= V^i\gamma^s_{\,\,ji}-V^lV^ i\gamma_{\,\,li}^p g^{ks}C_{pjk}:=N_{\,\,j}^s,
\end{equation}
where the quantities $N_{\,\,j}^s$ are the coefficients of the nonlinear connection associated to $L$ (see \cite[Eq. (2.3.2a)]{BaChSh00}).
Finally, using the last expression and \eqref{firstexpr},
\begin{equation}\label{gamma}
\Gamma^s_{\,\,ji}=\gamma^s_{\,\, ji}+g^{ks}\left(-N_{\,\,i}^p C_{pjk}-N_{\,\,j}^p C_{pki}+N_{\,\,k}^p C_{pij}\right).
\end{equation}
It is clear that Christoffel symbols depend just on the vector $V(x)$ and not in the vector field $V$, since they do not depend on the derivatives of $V$. This allows one to define a covariant derivative along a curve $\gamma$ by fixing a vector field $W$ along the curve. In order to check \eqref{almostmetric}, observe that if $\dot\gamma(t)\not=0$, then $D_\gamma^WX=\nabla^{\tilde{W}}_{\dot\gamma} \tilde{X}$ for any extensions $\tilde{X}$ and $\tilde{W}$ of $X$ and $W$ and \eqref{almostmetric} follows from the definition of $\nabla^{\tilde{W}}$. Assume now that $\dot\gamma(t)=0$. First observe that 
$D_\gamma^WZ(t)=\frac{dZ^i}{dt}(t)\left.\frac{\partial}{\partial x^i}\right|_{\gamma(t)}$ for any $Z\in{\mathfrak X}(\gamma)$ of coordinates $Z^1,\ldots,Z^n$. Then 
\begin{multline*}
\frac{d}{dt}(g_W(X,Y))=\frac{d}{dt}(X^iY^j g_{ij}(W))\\
=\frac{dX^i}{dt}Y^j g_{ij}(W)+X^i\frac{d Y^j }{dt}g_{ij}(W)+X^iY^j\frac{dg_{ij}(W)}{dt}\\=g_W(D_\gamma^WX,Y)+g_W(X,D_\gamma^WY)+X^iY^j\frac{dW^k}{dt}\frac{\partial g_{ij}}{\partial y^k}(W)\\
=g_W(D_\gamma^WX,Y)+g_V(X,D_\gamma^WY)+2C_V(D_\gamma^WW,X,Y),
\end{multline*}
as required. To check that the definition does not depend on the system of coordinates is left to the author (see also Remark \ref{sintetico})
\end{proof}
\begin{rem}\label{sintetico}
Observe that the covariant derivative along $\gamma$ with reference an $L$-admissible vector field $V\in{\mathfrak X}(\gamma)$ which has been  defined in Proposition \ref{curveconnection} can be also defined as the unique map  $D_\gamma^V:{\mathfrak X}(\gamma)\rightarrow {\mathfrak X}(\gamma)$, such that
\begin{enumerate}[(i)]
\item $D_\gamma^V(aZ_1+bZ_2)=a D_\gamma^VZ_1+bD_\gamma^VZ_2$, for $Z_1,Z_2\in{\mathfrak X}(\gamma)$ and $a,b\in \R$,
\item $D_\gamma^V(hZ)=\frac{dh}{dt}Z+hD_\gamma^VZ$, for $Z\in{\mathfrak X}(\gamma)$ and $h\in {\mathcal F}([a,b])$,
\item $D_\gamma^V X(\gamma)=\nabla^V_{\dot\gamma}X$ for $t\in[a,b]$ and $X\in{\mathfrak X}(\Omega)$,
\end{enumerate}
where ${\mathcal F}([a,b])$ is the subset of smooth real functions on $[a,b]$, in $\rm (iii)$ we consider any $L$-admissible extension of $V$ to an open subset $\Omega$ and 
$X(\gamma)$ is the vector field along $\gamma$ defined as $X(t)=X(\gamma(t))$ for every $t\in [a,b]$.
See also \cite[Proposition 3.18]{oneill}.
\end{rem}

\section{Curvature}\label{curvature}
 Along this section we will fix a pseudo-Finsler manifold $(M,L)$ and an $L$-admissible vector field $V$ defined in an open subset $\Omega\subset M$, being $\nabla^V$ the Chern connection of $(M,L)$ having $V$ as a reference vector field. We can define now the curvature associated to the affine connection $\nabla^V$ as a tensor $(1,3)$ in the open subset $\Omega\subset M$ defined by
\begin{equation*}
R^V(X,Y)Z=\nabla^V_X\nabla^V_YZ-\nabla^V_Y\nabla^V_XZ-
\nabla^V_{[X,Y]}Z
\end{equation*}
for every $X,Y,Z\in {\mathfrak X} (\Omega)$. It is straightforward to check that 
$R^V$ is a tensor. The curvature tensor satisfies some symmetries with respect to 
the metric $g_V$. We will need the covariant derivative of the Cartan tensor to 
express these symmetries. This covariant derivative $\nabla^VC_V$ is a $(0,4)$ tensor defined as
\begin{multline*}
\nabla^V_XC_V(Y,Z,W)=X(C_V(Y,Z,W))-C_V(\nabla^V_XY,Z,W)\\-C_V(Y,\nabla^V_XZ,W)-C_V(Y,W,\nabla^V_XW),
\end{multline*}
for every $X,Y,Z,W\in {\mathfrak X} (\Omega)$. It follows easily that $\nabla^V_XC_V$ 
is trilinear,  symmetric and 
\begin{equation}\label{formulilla}
\nabla^V_XC_V(V,Z,W)=-C_V(\nabla^V_XV,Z,W).
\end{equation}
\begin{prop}\label{propcurvature}
Let $X,Y,Z,W\in {\mathfrak X} (\Omega)$, then
\begin{enumerate}[(i)]
\item  $R^V(X,Y)=-R^V(Y,X)$,
\item $g_V(R^V(X,Y)Z,W)+g_V(R^V(X,Y)W,Z)=2B^V(X,Y,Z,W)$, where
\begin{multline*}
B^V(X,Y,Z,W)=\nabla^V_YC_V(\nabla_X^VV,Z,W)
-\nabla^V_XC_V(\nabla_Y^VV,Z,W)\\+
C_V(R^V(Y,X)V,Z,W),
\end{multline*}
\item $R^V(X,Y)Z+R^V(Y,Z)X+R^V(Z,X)Y=0$.
\end{enumerate}
Furthermore,
\begin{multline}\label{seisB}
g_V(R^V(X,Y)Z,W)-g_V(R^V(Z,W)X,Y)=\\
 B^V(Z,Y,X,W)+B^V(X,Z,Y,W)+B^V(W,X,Z,Y)\\
 +B^V(Y,W,Z,X)+B^V(W,Z,X,Y)+B^V(X,Y,Z,W).
\end{multline}
\end{prop}
\begin{proof}
 As the identities are tensorial,
 we can assume 
that  the brackets beween all the vector fields (excluding $V$) are zero.
The first identity follows immediately. For the second one, using the definition of $R^V$ and that $\nabla^V$ is almost metric $g$-compatible, we get
\begin{multline*}
g_V(R^V(X,Y)Z,W)+g_V(R^V(X,Y)W,Z)=\\
g_V(\nabla^V_X \nabla^V_YZ-\nabla^V_Y\nabla^V_XZ,W)
+g_V(\nabla^V_X\nabla^V_YW-\nabla^V_Y\nabla^V_XW,Z)\\
=X(g_V(\nabla^V_YZ,W))-g_V(\nabla^V_YZ,\nabla^V_XW)-2C_V(\nabla^V_XV,\nabla^V_YZ,W)\\
-Y(g_V(\nabla^V_XZ,W))+g_V(\nabla^V_XZ,\nabla^V_YW)+2C_V(\nabla^V_YV,\nabla^V_XZ,W)\\
+X(g_V(\nabla^V_YW,Z))-g_V(\nabla^V_YW,\nabla^V_XZ)-2C_V(\nabla^V_XV,\nabla^V_YW,Z)\\
-Y(g_V(\nabla^V_XW,Z))+g_V(\nabla^V_XW,\nabla^V_YZ)+2C_V(\nabla^V_YV,\nabla^V_XW,Z)\\
=X(g_V(\nabla^V_YZ,W)+g_V(\nabla^V_YW,Z))-Y(g_V(\nabla^V_XZ,W)+g_V(\nabla^V_XW,Z))\\
-2C_V(\nabla^V_XV,\nabla^V_YZ,W)+2C_V(\nabla^V_YV,\nabla^V_XZ,W)
\\-2C_V(\nabla^V_XV,\nabla^V_YW,Z)+2C_V(\nabla^V_YV,\nabla^V_XW,Z)\\
=X(Y(g_V(Z,W))-2C_V(\nabla^V_YV,Z,W))-Y(X(g_V(W,Z))-2C_V(\nabla^V_XV,W,Z))\\
-2C_V(\nabla^V_XV,\nabla^V_YZ,W)+2C_V(\nabla^V_YV,\nabla^V_XZ,W)
\\-2C_V(\nabla^V_XV,\nabla^V_YW,Z)+2C_V(\nabla^V_YV,\nabla^V_XW,Z)\\
=[X,Y](g_V(Z,W))+2(-\nabla^V_XC_V(\nabla^V_YV,W,Z)-C_V(\nabla^V_X\nabla^V_YV,Z,W)
\\-C_V(\nabla^V_YV,\nabla^V_XZ,W)-C_V(\nabla^V_YV,Z,\nabla^V_XW)+\nabla^V_YC_V(\nabla^V_XV,W,Z)\\+C_V(\nabla^V_Y\nabla^V_XV,Z,W)
+C_V(\nabla^V_XV,\nabla^V_YZ,W)+C_V(\nabla^V_XV,Z,\nabla^V_YW)
\\-C_V(\nabla^V_XV,\nabla^V_YZ,W)+C_V(\nabla^V_YV,\nabla^V_XZ,W)
\\-C_V(\nabla^V_XV,\nabla^V_YW,Z)+C_V(\nabla^V_YV,\nabla^V_XW,Z))\\
=2(\nabla^V_YC_V(\nabla^V_XV,W,Z))-\nabla^V_XC_V(\nabla^V_YV,W,Z)+C_V(R^V(Y,X)V,Z,W),
\end{multline*}
as we wanted to prove (since $C_V$ and $\nabla^V_EC_V$ are symmetric for any $E\in {\mathfrak X}(\Omega)$). The third identity is true for any torsion-free connection (see for example the proof in \cite[Proposition 3.36]{oneill}).
To check  \eqref{seisB}, use the third identity to deduce the
following four ones,
\begin{align*}
g_V(R^V(X,W)Y+R^V(W,Y)X+R^V(Y,X)W,Z)&=0,\\
g_V(R^V(X,Y)Z+R^V(Y,Z)X+R^V(Z,X)Y,W)&=0.\\
g_V(R^V(X,W)Z+R^V(W,Z)X+R^V(Z,X)W,Y)&=0,\\
g_V(R^V(Y,Z)W+R^V(Z,W)Y+R^V(W,Y)Z,X)&=0.
\end{align*}
Then summing up the four identities and using the symmetries of parts $(i)$ and
 $(ii)$, it comes out
\begin{multline*}
2g_V(R^V(X,Y)Z,W)-2g_V(R^V(Z,W)X,Y)+\\
2B^V(Y,Z,X,W)+2B^V(Z,X,Y,W)+2B^V(X,W,Z,Y)\\
+2B^V(W,Y,Z,X)+2B^V(Z,W,X,Y)+2B^V(Y,X,Z,W)=0.
\end{multline*} 
 Taking into account that $B^V$ is anti-symmetric in the two first components and symmetric in the two last ones
  we get \eqref{seisB}.
\end{proof}
Having at hand the affine connection $\nabla^V$ we can compute the
derivative of any tensor. In particular,
\begin{multline*}
\nabla^V_XR^V(Y,Z)W=\nabla^V_X(R^V(Y,Z)W)\\-R^V(\nabla^V_XY,Z)W
-R^V(Y,\nabla^V_XZ)W-R^V(Y,Z)(\nabla^V_XW)
\end{multline*}
for every $X,Y,Z,W\in {\mathfrak X} (\Omega)$. As $\nabla^V$ is an affine connection, $R^V$ also satisfies the Second Bianchi identity (see for example the proof in \cite[Proposition 3.36]{oneill}).
\subsection{Two parameters maps}\label{twoparam}
Let ${\mathcal D}$ be an open subset of $\R^2$ satisfying the interval condition, namely, horizontal and vertical lines of $\R^2$ intersect $\mathcal D$ in intervals. A two-parameter map is a smooth map
$\Lambda:{\mathcal D}\rightarrow M$. We will use the following notation:
\begin{enumerate}
\item the $t$-parameter curve of $\Lambda$
in $s_0$ is the curve $\gamma_{s_0}$ defined as $t\rightarrow \gamma_{s_0}(t)=\Lambda(t,s_0)$
\item the $s$ parameter curve of $\Lambda$ in $t_0$ is the curve $\beta_{t_0}$ defined as $s\rightarrow \beta_{t_0}(s)=\Lambda(t_0,s)$.
\end{enumerate}
Moreover, we will denote by $\Lambda_t(t,s)=\dot\gamma_s(t)$ and 
$\Lambda_s(t,s)=\dot\beta_t(s)$. Let us define $\Lambda^*(TM)$ as the vector bundle over ${\mathcal D}$ induced by $\pi:TM\rightarrow M$ through $\Lambda$. Then we denote the subset of smooth sections 
of $\Lambda^*(TM)$ as ${\mathfrak X}(\Lambda)$. Observe that a vector field $V\in {\mathfrak X}(\Lambda)$ induces vector fields in $\mathfrak X(\gamma_{s_0})$ and $\mathfrak X(\beta_{t_0})$.
We will say that $V$ is $L$-admissible if $V(t,s)\in A$ for every $(t,s)\in \mathcal D$.
  When $\Lambda$ lies in the domain of a coordinate system $x^1,\ldots,x^n$, we will denote $\Lambda^i=x^i\circ\Lambda$.
  \begin{prop}\label{commut}
 With the above notation, if   $V\in {\mathfrak X}(\Lambda)$ is $L$-admissible, then $D_{\gamma_s}^V{\dot\beta_t}=D_{\beta_t}^V{\dot\gamma_s}$.
  \end{prop}
  \begin{proof}
  Using Proposition \ref{curveconnection} we get
\begin{align*}
  D_{\gamma_s}^V{\dot\beta_t}=\left(\frac{\partial^2\Lambda^k}{\partial t\partial s}+\Gamma_{\,\,ij}^k(V)\frac{\partial\Lambda^i}{\partial s}\frac{\partial\Lambda^j}{\partial t}\right) \frac{\partial}{\partial x^k},\\
  D_{\beta_t}^V{\dot\gamma_s}=\left(\frac{\partial^2\Lambda^k}{\partial s\partial t}+\Gamma_{\,\,ij}^k(V)\frac{\partial\Lambda^i}{\partial t}\frac{\partial\Lambda^j}{\partial s}\right) \frac{\partial}{\partial x^k}.
  \end{align*}
  Both quantities coincide because $\Gamma_{\,\,ij}^k(V)$ is symmetric in $i,j$ and $\frac{\partial^2\Lambda^k}{\partial s\partial t}=\frac{\partial^2\Lambda^k}{\partial t\partial s}$.
  \end{proof}
\subsection{Jacobi operator and flag curvature}\label{flag}
In a fixed point $p\in M$, the curvature tensor $R^V$ depends not only on $V(p)$ but on the extension V. Let us see that the quantity $R^{V}(V,U)W$ depends only on the value of $V$  along the integral curve of $V$.
\begin{prop}\label{curvaturacurva}
Let $(M,L)$ be a pseudo-Finsler manifold, $\gamma:(a-\varepsilon,a+\varepsilon)\rightarrow M$ an $L$-admissible smooth curve
and $u,w\in T_{\gamma(a)}M$. If $V$ is an $L$-admissible extension of $\dot\gamma$ 
and $U$ and $W$ extensions of  $u$ and $w$, then
\[R^{\gamma}(\dot\gamma(a),u)w:=R^V_{\gamma(a)}(V,U)W\]
 is well-defined, namely, it does not depend on the extensions used to compute it.
\end{prop}
\begin{proof}
As the result is local, we can assume that the image of $\gamma$ is contained in an open subset $\Omega$ that admits a system of coordinates $(\Omega,\varphi)$. First assume that $V$ and $U$  are the variational vector fields of the two-parametric variation of $\gamma$,
$\Lambda:(a-\varepsilon,a+\varepsilon)\times (-\varepsilon_1,\varepsilon_1)\rightarrow M$, $(t,s)\rightarrow \Lambda(t,s)$, namely, 
$V(\Lambda(t,s))=\Lambda_t(t,s)$ and $U(\Lambda(t,s))=\Lambda_s(t,s)$  for every $(t,s)\in (a-\varepsilon,a+\varepsilon)\times (-\varepsilon_1,\varepsilon_1)$ and the image of $\Lambda$ lies in $\Omega$ (recall notation in Subsection \ref{twoparam}). We can also assume that the curves $\gamma_s$ are $L$-admissible for $s\in (-\varepsilon_1,\varepsilon_1)$  by taking $\varepsilon_1$ small enough and
 that $W\in{\mathfrak X}(\Lambda)$. We will denote by $W^i$ the coordinates of $W$ in $(\Omega,\varphi)$, being $W^i_t$ and $W^i_s$ the partial derivatives with respect to the parameters of the variation $t$ and $s$. Then using \eqref{connection} twice we get
\begin{multline*}
D_{\beta_{t}}^{\Lambda_t}D_{\gamma_{s}}^{\Lambda_t}W=
\left[W_{ts}^k+W_{s}^i\Lambda_t^j \Gamma_{\,\,ij}^k(\Lambda_t)
+W^i\Lambda_{ts}^j \Gamma_{\,\,ij}^k(\Lambda_t)\right.\\\left.
+W^i\Lambda_t^j \frac{\partial}{\partial s}\Gamma_{\,\,ij}^k(\Lambda_t)
+W_{t}^l\Lambda_s^m \Gamma_{\,\,lm}^k(\Lambda_t)\right.\\\left.+W^i\Lambda_t^j \Lambda_s^m \Gamma_{\,\,ij}^l(\Lambda_t) \Gamma_{\,\,lm}^k(\Lambda_t)\right]\frac{\partial}{\partial x^k}
\end{multline*}
and 
\begin{multline*}
D_{\gamma_{s}}^{\Lambda_t}D_{\beta_{t}}^{\Lambda_t}W=
\left[W_{st}^k+W_{t}^i\Lambda_s^j \Gamma_{\,\,ij}^k(\Lambda_t)
+W^i\Lambda_{st}^j \Gamma_{\,\,ij}^k(\Lambda_t)\right.\\\left.
+W^i\Lambda_s^j \frac{\partial}{\partial t}\Gamma_{\,\,ij}^k(\Lambda_t)
+W_{s}^l\Lambda_t^m \Gamma_{lm}^k(\Lambda_t)\right.\\\left.+W^i\Lambda_s^j \Lambda_t^m \Gamma_{\,\,ij}^l(\Lambda_t) \Gamma_{\,\,lm}^k(\Lambda_t)\right]\frac{\partial}{\partial x^k}.
\end{multline*}
 Then
 \begin{multline}\label{curvatura}
 D_{\gamma_{s}}^{\Lambda_t}
 D_{\beta_{t}}^{\Lambda_t}W-D_{\beta_{t}}^{\Lambda_t}D_{\gamma_{s}}^{\Lambda_t}W=
\left[W^i\Lambda_s^j \frac{\partial}{\partial t}\Gamma_{\,\,ij}^k(\Lambda_t)-W^i\Lambda_t^j \frac{\partial}{\partial s}\Gamma_{\,\,ij}^k(\Lambda_t)\right.\\\left.
 +W^i\Lambda_s^j \Lambda_t^m \left(\Gamma_{\,\,ij}^l(\Lambda_t) \Gamma_{\,\,lm}^k(\Lambda_t)-\Gamma_{\,\,im}^l(\Lambda_t) \Gamma_{\,\,lj}^k(\Lambda_t)\right)\right]\frac{\partial}{\partial x^k}.
 \end{multline}
Define $\tilde{\Gamma}_{\,\,ij}^l(p)=\Gamma_{\,\,ij}^l(V(p))$ for every $p\in\Omega$ and observe that 
$\tilde{\Gamma}_{\,\,ij}^l$ are the Christoffel symbols of the affine connection $\nabla^V$ in $\Omega$. In particular, we have that  $\tilde{\Gamma}_{\,\,ij}^l(\Lambda)=\Gamma_{\,\,ij}^l(\Lambda_t)$. Taking into account that $[V,U]=0$ (at least in the points $\Lambda(t,s)$),
 we get
 \begin{multline*}
 R^V(V,U)W=D^V_VD_U^VW-D^V_UD_V^VW=\\
 \left[W^i\Lambda_s^j\Lambda_t^p \frac{\partial\tilde{\Gamma}_{\,\,ij}^k}{\partial x^p}
 (\Lambda)-W^i\Lambda_t^j \Lambda_s^p\frac{\partial \tilde{\Gamma}_{\,\,ij}^k}{\partial x^p}(\Lambda)\right.\\\left.
 +W^i\Lambda_s^j \Lambda_t^m \left(\tilde{\Gamma}_{\,\,ij}^l(\Lambda) \tilde{\Gamma}_{\,\,lm}^k(\Lambda)-\tilde{\Gamma}_{\,\,im}^l(\Lambda) \tilde{\Gamma}_{\,\,lj}^k(\Lambda)\right)\right]\frac{\partial}{\partial x^k}.
 \end{multline*}
 As $\tilde{\Gamma}_{\,\,ij}^l(\Lambda)=\Gamma_{\,\,ij}^l(\Lambda_t)$, $\frac{\partial}{\partial t}\Gamma_{\,\,ij}^k(\Lambda_t)=\Lambda_t^p\frac{\partial \tilde{\Gamma}_{\,\,ij}^k}{\partial x^p}(\Lambda)$ and $\frac{\partial}{\partial s}\Gamma_{\,\,ij}^k(\Lambda_t)=\Lambda_s^p\frac{\partial \tilde{\Gamma}_{\,\,ij}^k}{\partial x^p}(\Lambda)$, we conclude that 
 \[R^V(V,U)W=D_{\gamma_{s}}^{\Lambda_t}
 D_{\beta_{t}}^{\Lambda_t}W-D_{\beta_{t}}^{\Lambda_t}D_{\gamma_{s}}^{\Lambda_t}W\]
 as required.
 Moreover,
 \begin{align}\label{firstpartial}
 \frac{\partial}{\partial t}\Gamma_{\,\,ij}^k(\Lambda_t)&=\Lambda_t^p \frac{\partial \Gamma_{\,\,ij}^k}{\partial x^p}(\Lambda_t)+\Lambda_{tt}^p \frac{\partial \Gamma_{\,\,ij}^k}{\partial y^p}(\Lambda_t),\\
 \label{secondpartial}
 \frac{\partial}{\partial s}\Gamma_{\,\,ij}^k(\Lambda_t)&=\Lambda_s^p \frac{\partial \Gamma_{\,\,ij}^k}{\partial x^p}(\Lambda_t)+\Lambda_{ts}^p \frac{\partial \Gamma_{\,\,ij}^k}{\partial y^p}(\Lambda_t).
 \end{align}
Here recall that $x^1,\ldots,x^n,y^1,\ldots,y^n$ is the natural coordinate system of $T\Omega$ associated to the coordinate system $(\Omega,\varphi)$. From the above equations and \eqref{curvatura}, it follows that $R^V(V,U)W$ depends only on the curve $\gamma$ and the values of the vector fields $U$ and $W$ along $\gamma$. Now given any vector field $V$ extending $\dot\gamma$, observe that the value of $R^V(V,U)W$ does not depend on the extensions  $U$ and $W$ of the vectors $u,w\in T_{\gamma(a)}M$. It is always possible to get extensions such that $V$ and $U$ are the variational vector fields of a two-parametric map and $W$ a smooth vector field on it. Indeed, consider a system of coordinates adapted to $V$ in a neighborhood $\Omega$ of $\gamma(a)$ small enough, in the sense that $V=\frac{\partial}{\partial x^1}$ in this neighborhood. If 
 $u=a^i \left.\frac{\partial}{\partial x^i}\right|_{\gamma(a)}$ and 
 $w= b^i \left.\frac{\partial}{\partial x^i}\right|_{\gamma(a)},$ then  $U= a^i \frac{\partial}{\partial x^i}$ and $W= b^i \frac{\partial}{\partial x^i}$ are the required extensions in $\Omega\subset M$.
\end{proof}
Let us recall that the Chern connection can also be interpreted as a connection in the fiber bundle $\pi^*_A:\pi_A^*(TM)\rightarrow A$ (see for example \cite[Remark 2.5]{JavSoares13}) and we can define the curvature 2-forms associated to this connection. In particular, the horizontal part of these 2-forms is the so-called $hh$-curvature tensor (see \cite[Chapter 3]{BaChSh00}), which in coordinates is written as
\begin{equation}\label{curvchern}
R_v(V,U)W=V^j U^k W^l R_{j\,\,kl}^{\,\,\,i}(v) \frac{\partial}{\partial x^i}
\end{equation}
for $v\in A$ and $V=V^i \frac{\partial}{\partial x^i},U= U^i \frac{\partial}{\partial x^i},W= W^i \frac{\partial}{\partial x^i}$  vector fields in $\Omega\subset M$, where 
\begin{multline*}
R_{j\,\,kl}^{\,\,\,i}(v)= \frac{\partial \Gamma_{\,\,jl}^i}{\partial x^k}(v)-N_{\,\,k}^p(v) \frac{\partial \Gamma_{\,\,jl}^i}{\partial y^p}(v)-
\frac{\partial \Gamma_{\,\,jk}^i}{\partial x^l}(v)+N_{\,\,l}^p(v) \frac{\partial \Gamma_{\,\,jk}^i}{\partial y^p}(v)\\
 +\Gamma_{\,\,hk}^i(v)\Gamma_{\,\,jl}^h(v)-\Gamma_{\,\,hl}^i(v)\Gamma_{\,\,jk}^h(v),
\end{multline*}
and $N_{\,\,k}^p$ has been defined in \eqref{vi} (see \cite[Formula (3.3.2) and Exercise 3.9.6]{BaChSh00}). Now given a system of coordinates $(\Omega,\varphi)$, for every $L$-admissible smooth curve $\gamma:[a,b]\rightarrow \Omega\subset M$,  define 
\[H_\gamma (U,W)=U^i W^j (D_\gamma^{\dot\gamma}\dot\gamma)^p \frac{\partial\Gamma_{\,\,ij}^k}{\partial y^p}(\dot\gamma)\frac{\partial}{\partial x^k}.\]
It is easy to see that the definition does not depend on the choice of coordinates and then we can define a symmetric tensor $H_\gamma:{\mathfrak X}(\gamma)\times {\mathfrak X}(\gamma)\rightarrow {\mathfrak X}(\gamma)$ for every $L$-admissible smooth curve $\gamma:[a,b]\rightarrow M$ that does not lie necessarily in a domain of coordinates.
\begin{thm}\label{curvaturaway}
Let $(M,L)$ be a pseudo-Finsler manifold. Consider an $L$-ad\-mis\-sible smooth curve $\gamma:(a-\varepsilon,a+\varepsilon)\rightarrow M$. With the above notation
\begin{equation}\label{curformula}
R^\gamma(\dot\gamma(a),u)w=R_{\dot\gamma(a)}(\dot\gamma(a),u)w+H_\gamma(u,w)
\end{equation}
for any $u,w\in T_{\gamma(a)}M$.
\end{thm}
\begin{proof}
First observe that we can choose any extension of $u$ to the curve $\gamma$. In particular, we can choose a parallel vector field $U=U^i(t) \left.\frac{\partial}{\partial x^i}\right|_{\gamma(t)}$ along $\gamma$, which satisfies
$\frac{d U^k}{dt}=-U^i\dot\gamma^j\Gamma_{\,\, ij}^k(\dot\gamma)$. Moreover,  $\frac{d \dot\gamma^k}{dt}=(D_\gamma^{\dot\gamma}\dot\gamma)^k-\dot\gamma^i\dot\gamma^j\Gamma_{\,\, ij}^k(\dot\gamma)$. Let $\Lambda$ be a two-parametric  variation of $\gamma$ such that $\dot\beta_t=U$ and $\gamma_s$ are $L$-admissible curves (recall notation of Subsection \ref{twoparam}). As $\Lambda_{tt}^p=\frac{d \dot\gamma^p}{dt}$ and 
$\Lambda_{ts}^p=\frac{d U^p}{dt}$, substituting the last formulae in \eqref{firstpartial} and \eqref{secondpartial}, these equations in \eqref{curvatura} and making $s=0$, we obtain
\begin{multline*}
R^\gamma(\dot\gamma,u)w=\left[u^i w^j\left(\dot\gamma^ p\frac{\partial \Gamma_{\,\,ij}^k}{\partial x^p}(\dot\gamma)-\dot\gamma^l\dot\gamma^m
\Gamma_{\,\,lm}^p(\dot\gamma) \frac{\partial \Gamma_{\,\,ij}^k}{\partial y^p}(\dot\gamma)\right)\right.\\
-w^i \dot\gamma^j\left(
u^p\frac{\partial \Gamma_{\,\,jk}^i}{\partial x^p}(\dot\gamma)+u^l\dot\gamma^m
\Gamma_{\,\,lm}^p(\dot\gamma) \frac{\partial \Gamma_{\,\,ij}^k}{\partial y^p}(\dot\gamma)\right)\\
 +w^i u^j \dot\gamma^m\left(\Gamma_{\,\,ij}^l(\dot\gamma)\Gamma_{\,\,lm}^k(\dot\gamma)-\Gamma_{\,\,im}^l(\dot\gamma)\Gamma_{\,\,lj}^k(\dot\gamma)\right)\left]\frac{\partial}{\partial x^k}\right.
 +H_\gamma(u,w).
 \end{multline*}
 Observing that $N_{\,\, l}^p(\dot\gamma)=\dot\gamma^m
\Gamma_{\,\,lm}^p(\dot\gamma)$ (see \eqref{vi}), we get \eqref{curformula}.
\end{proof}
This theorem can be used to compute the flag curvature using $\nabla^V$ and $R^V$.
\begin{cor}\label{flagcur}
Given a plane $\pi={\rm span}\{v,u\}$ which is $g_v$-nondegenerate,   the quantity
\[K_v(u)=\frac{g_v(R^{\gamma_v}(v,u)u,v)}{L(v)g_v(u,u)-g_v(v,u)^2},\]
where $\gamma_v$ is the geodesic with velocity $v$ at $t=0$, is the flag curvature of $\pi$ with flagpole $v$.
\end{cor}
\begin{proof}
After Theorem \ref{curvaturaway}, it is straightforward that this quantity
is the flag curvature for the pseudo-Finsler metric $L$ of $\pi$ with flagpole $v$ (see \cite[Section 3.9]{BaChSh00}), since $H_{\gamma_v}=0$ because $\gamma_v$ is a geodesic, that is, $D_{\gamma_v}^{\dot\gamma_v}\dot\gamma_v=0$.
\end{proof}
Finally,
observe that,  with the notation of the corollary, the quantity 
\[K_v(u,w)=\frac{g_v(R^{\gamma_v}(v,u)w,v)}{L(v)g_v(u,w)-g_v(v,u)g_v(v,w)}\]
is the {\it predecessor of the flag curvature} (see \cite[page 69]{BaChSh00}).

\section*{Acknowledgments}

 I would like to warmly acknowledge A. R. Mart\'inez and  Professors M. Alexandrino, E. Caponio, 
and M. S\'anchez for a critical reading of some of the results contained in this manuscript.

\end{document}